\documentclass{article}
\usepackage[pdftex]{graphicx}
  \graphicspath{{./pics/}}
  \DeclareGraphicsExtensions{.pdf}
\usepackage{amsmath}
\usepackage{amsthm}
\usepackage[caption=false,font=footnotesize]{subfig}

\def\qedd{\relax\ifmmode\hskip2em \diamond\else\unskip\nobreak\hskip1em $\diamond$\fi}

\newtheorem{theorem}{Theorem}
\newtheorem{lem}[theorem]{Lemma}
\newtheorem{cor}[theorem]{Corollary}
\newtheorem{definition}[theorem]{Definition}
\newtheorem{prob}[theorem]{Problem}
\newtheorem{exmp}[theorem]{Example}
\newtheorem{rem}[theorem]{Remark}

\newcommand{\eps}{\varepsilon}
\newcommand{\supc}{\mbox{$\sup {\rm C}$}}
\newcommand{\supC}{\mbox{$\sup {\rm C}$}}
\newcommand{\supcc}{\mbox{$\sup {\rm cC}$}}
\newcommand{\supCC}{\mbox{$\sup {\rm cC}$}}

\begin{document}
\title{Supervisory Control Synthesis of Discrete-Event Systems using Coordination Scheme}

\markboth{IEEE Transactions on automatic control,~Vol.~6, No.~1, January~2007}%
{J. Komenda \MakeLowercase{\textit{et al.}}:
Supervisory Control Synthesis of Discrete-Event Systems using Coordination Scheme}

\author{Jan~Komenda$^a$, Tom{\' a}{\v s}~Masopust$^{a,b}$, and~Jan~H.~van~Schuppen$^b$\\[8pt]
    $^a$\small Institute of Mathematics, Czech Academy of Sciences\\
        \small {\v Z}i{\v z}kova 22, 616 62 Brno, Czech Republic\\
        \small \texttt{komenda{@}ipm.cz, masopust{@}ipm.cz}\\[8pt]
    $^b$\small CWI, P.O. Box 94079, 1090 GB Amsterdam, The Netherlands\\
        \small \texttt{T.Masopust@cwi.nl, J.H.van.Schuppen@cwi.nl}\\[8pt]
}
\date{Technical Report}
\maketitle

\begin{abstract}
  Supervisory control of discrete-event systems with a global safety specification and with only local supervisors is a difficult problem. For global specifications the equivalent conditions for local control synthesis to equal global control synthesis may not be met. This paper formulates and solves a control synthesis problem for a generator with a global specification and with a combination of a coordinator and local controllers. Conditional controllability is proven to be an equivalent condition for the existence of such a coordinated controller. A procedure to compute the least restrictive solution is also provided in this paper and conditions are stated under which the result of our procedure coincides with the supremal controllable sublanguage.
\end{abstract}

\section{Introduction}
  This paper investigates the supervisory control synthesis of modular discrete-event systems with a coordinator. Discrete-event systems (DES) represented as finite-state machines have been studied by P.~J.~Ramadge and W.~M.~Wonham in \cite{RW89}. Large discrete-event systems are typically formed as a synchronous composition of a large number of local components (subsystems) that are themselves modeled by finite-state machines and run in parallel. Systems formed in this way are often called modular discrete-event systems.

  The aim of supervisory control is to ensure that the control objectives of safety and of liveness are satisfied by the closed-loop system. Specifically, the safety property means that the behavior (language) of the system must be included in a specified language, called a specification, and the liveness property means that the system cannot get to deadlock or livelock. Since only so-called controllable specification languages can be achieved, one of the key issues in supervisory control synthesis is the computation of the supremal controllable sublanguage of the given specification, from which the supervisor can then be constructed.

  From an application viewpoint, global (indecomposable) specifications are much more interesting than local specifications. Sometimes, local subsystems are independent (in the sense that their event sets are disjoint), and they are only coupled implicitly via a global specification. In the case of global specifications, it is often impossible to synthesize the supervisors locally, i.e., within a fully decentralized control architecture. In some cases it is possible to exploit the modular structure of the plant and to avoid the manipulation with the global plant. However, structural conditions on local plant languages proposed in \cite{GM04} and further weakened in \cite{KvSGM08} under which this is possible are still very restrictive.

  In this paper, another approach to deal with global specifications is introduced. It relies on the coordination control scheme proposed first in \cite{KvS08}, where a coordinator is applied for the control of modular discrete-event systems. The coordinator receives a part of the observations (events) from local subsystems and its task is to satisfy the global part of the specification and the nonblockingness. Hence, the coordinator can be seen as a two-way communication channel, where some events belonging to the coordinator event set are exchanged (communicated) between both subsystems.

  Thus, coordination control may be seen as a reasonable trade-off between a purely decentralized control synthesis, which is in some cases unrealistic, and a global control synthesis, which is naturally prohibitive for space complexity reasons. Moreover, the conditions obtained from the coordination control framework are based  on the specification itself rather than on local plants.

  In this paper, we are only concerned with the safety issue. First, we propose a necessary and sufficient condition on a specification language to be exactly achieved in the coordination control architecture that consists of a coordinator, its supervisor, and local supervisors for the subsystems. We call this condition {\em conditional controllability}, and it refines the condition that was only a sufficient one and has been presented in \cite{KvS08}. It is shown that the supremal conditionally controllable sublanguage of a given specification language always exists. In addition to the above mentioned existential result, a procedure for computation of the supremal conditionally controllable sublanguage is proposed. Finally, in the setting of this computation procedure the supremal conditionally controllable sublanguage is shown to be included in the supremal  controllable sublanguage and additional conditions are found under which both concepts coincide.

  The organization of this paper is as described below. In the next section, decentralized supervisory control of modular discrete-event systems is recalled and the coordination control approach is motivated. In Sections \ref{sec:cc} and \ref{sec:concepts} we briefly recall the coordination control framework and concepts. In Section~\ref{sec:controlsynthesis}, our first result is presented: the equivalence condition on a specification language to be exactly achieved in the coordination control architecture. In addition, we show that the supremal conditionally controllable sublanguage always exists. Then, in Section~\ref{sec:procedure}, a procedure for its computation is proposed. Finally, in Section \ref{sec:conclusion}, some concluding remarks are summarized including a discussion on future extensions of this work.

\section{Decentralized and coordination control\\ of modular discrete-event systems}\label{sec:cc}
  In this section, the elements of supervisory control theory needed in the rest of this paper are recalled. For more details, the reader is referred to lecture notes \cite{Won04} or the book \cite{CL08}. Discrete-event systems (DES) are modeled as deterministic generators that are finite-state machines with partial transition functions. A (deterministic) {\em generator\/} $G$ is a quintuple
  \[
    G=(Q,E,f,q_0,Q_m)\,,
  \]
  where $Q$ is a finite set of {\em states}, $E$ is the finite set of {\em events}, $f: Q \times E \to Q$ is the {\em partial transition function}, $q_0 \in Q$ is the {\em initial state}, and $Q_m\subseteq Q$ is the set of {\em marked states}. Recall that $f$ can be extended by induction to $f: Q \times E^* \to Q$ in the usual way. The behaviors of DES generators are defined in terms of languages. The {\em language} of $G$ is defined as $L(G) = \{s\in E^* \mid f(q_0,s)\in Q\}$, and the {\em marked language} of $G$ is defined as $L_m(G) = \{s\in E^* \mid f(q_0,s)\in Q_m\}$.

  The {\em natural projection} $P: E^* \to E_0^*$, for some $E_0\subseteq E$, is a mapping (morphism) which erases all symbols from $E\setminus E_0$ and keeps all the other symbols unchanged, i.e., it is defined so that
  \begin{itemize}
    \item $P(a)=\eps$, for $a\in E\setminus E_0$,
    \item $P(a)=a$, for $a\in E_0$, $P(\eps)=\eps$, and
    \item for $u,v\in E^*$,  $P(uv)=P(u)P(v)$.
  \end{itemize}

  The {\em inverse image} of $P$, denoted by $P^{-1} : E_0^* \to 2^{E^*}$, is defined as
  \[
    P^{-1}(a)=\{x\in E^* \mid P(x) = a\}\,.
  \]
  These definitions are naturally extended to languages.

  In what follows, given event sets $E_i$, $E_j$, $E_k$, we denote by $P^{i+j}_k$ the projection from $E_i\cup E_j$ to $E_k$, and by $P^{i}_{j\cap k}$ the projection from $E_i$ to $E_j \cap E_k$. In addition, denote $E_{i+j}=E_i \cup E_j$, for $i,j\in\{1,2,k\}$. Let $E_u\subseteq E$ be the set of uncontrollable events and denote by $E_{i,u}=E_u\cap E_i$, for $i=1,2,k$, the corresponding sets of locally uncontrollable events. Then, as mentioned above, $E_{i+j,u}$ denotes the set $E_{i+j}\cap E_u$.

  Below, modular DES are considered. First, we recall that the synchronous product (also called the parallel composition) of languages $L_1\subseteq E_1^*$ and $L_2\subseteq E_2^*$ is defined by
  \[
    L_1\| L_2=P_1^{-1}(L_1) \cap P_2^{-1}(L_2) \subseteq E^*\,,
  \]
  where $P_i: E^*\to E_i^*$, for $i=1,2$, are natural projections to local event sets.

  The synchronous product can also be defined for generators. In this case, for two generators $G_1$ and $G_2$, it is well known that $L(G_1 \| G_2) = L(G_1) \| L(G_2)$ and $L_m(G_1 \| G_2)= L_m(G_1) \| L_m(G_2)$. The reader is referred to~\cite{CL08} for more details.

  A {\em controlled generator\/} is a structure
  \[
    (G,E_c,\Gamma)\,,
  \]
  where $G$ is a generator, $E_c \subseteq E$ is the set of {\em controllable events}, $E_{u} = E \setminus E_c$ is the set of {\em uncontrollable events}, and
  \[
    \Gamma = \{\gamma \subseteq E \mid E_{u}\subseteq\gamma\}
  \]
  is the set of {\em control patterns}.

  A {\em supervisor} for the controlled generator $(G,E_c,\Gamma)$ is a map $S:L(G) \to \Gamma$.

  A {\em closed-loop system} associated with the controlled generator $(G,E_u,\Gamma)$ and the supervisor $S$ is defined as the smallest language $L(S/G) \subseteq E^*$ which satisfies
  \begin{enumerate}
    \item $\eps \in L(S/G)$,
    \item if $s \in L(S/G)$, $sa\in L(G)$, and $a \in S(s)$, then also $sa \in L(S/G)$.
  \end{enumerate}

  In the automata framework, where the supervisor is represented by a DES generator, the closed-loop system can be recast as a synchronous product of the supervisor and the plant because it follows from the form of the control patterns that the supervisor never disables uncontrollable events, i.e., all uncontrollable transitions are always enabled. This is known as {\em admissibility} of a supervisor. Hence, for an admissible supervisor $S$ that controls the plant $G$, one can write
  \[
    L(S/G) = L(S) \| L(G)\,.
  \]

  The prefix closure $\overline{L}$ of a language $L$ is the set of all prefixes of all its words. A language $L\subseteq E^*$ is said to be prefix-closed if $L=\overline{L}$.

  \begin{definition}
    Let $L$ be a prefix-closed language over an event set $E$ with the uncontrollable event set $E_u\subseteq E$. A language $K\subseteq E^*$ is {\em controllable} with respect to $L$ and $E_u$ if
    \[
      \overline{K}E_u\cap L\subseteq \overline{K}\,.
    \]
  \end{definition}

  Given a prefix-closed specification language $K\subseteq E^*$, the goal of supervisory control theory is to find a supervisor $S$ such that
  \[
    L(S/G)=K\,.
  \]
  It is known that such a supervisor exists if and only if $K$ is controllable \cite{RW89}.

  Thus, for specifications that are not controllable, controllable sublanguages are considered. The notation $\supc(K,L,E_u)$ is chosen for the supremal controllable sublanguage of $K$ with respect to $L$ and $E_u$. This supremal controllable sublanguage always exists and equals to the union of all controllable sublanguages of $K$, see e.g.~\cite{CL08}.

  A modular DES is simply a synchronous product of two or more generators. Decentralized control synthesis of a modular DES is a procedure, where the control synthesis is carried out for each module or local subsystem. The global supervisor then formally consists of the synchronous product of local supervisors although that product is not computed in practice. In terms of behaviors, the optimal global control synthesis is represented by the closed-loop language
  \[
    \supc (K,L,E_u) = \supc (\|_{i=1}^n K_i,\|_{i=1}^n L_i,E_u)\,.
  \]

  Given a rational global specification language $K\subseteq E^*$, one can theoretically always compute its supremal controllable sublanguage from which the optimal (least restrictive) supervisor can be built. Such a global control synthesis of a modular DES consists simply in computing the global plant and then the control synthesis is carried out as described above.

  Decentralized control synthesis means that the specification language $K$ is replaced by
  \[
    K_i=K\cap P_i^{-1}(L_i)
  \]
  and the synthesis is done similarly as for local specifications or using the notion of partial controllability \cite{GM04}. Note the difference with decentralized control of monolithic plants as studied in \cite{YLL02}. However, the purely decentralized control synthesis is not always possible as the sufficient conditions under which it can be used are quite restrictive. Therefore, we have proposed the coordination control in \cite{KvS08} as a trade-off between the purely decentralized control synthesis, which is in some cases unrealistic, and the global control synthesis, which is naturally prohibitive for complexity reasons.

\section{Concepts}\label{sec:concepts}
  Coordination control for DES is inspired by the concept of conditional independence of the theory of probability and of stochastic processes. Recall from \cite{KvS08} that conditional independence is roughly captured by the event set condition, when every joint action (move) of local subsystems must be accompanied by a coordinator action. In this paper, after the architecture of the coordination scheme is recalled, a new necessary and sufficient condition on a specification language to be exactly achieved in this architecture is presented.

  In the coordination scheme, first a supervisor $S_k$ for the coordinator is synthesized that takes care of the part $P_k(K)$ of the specification $K$. Then, supervisors $S_i$, for $i=1,2$, are synthesized so that the remaining parts of the specification, i.e., $P_{i+k}(K)$, are met by the new plant languages $G_i \| (S_k/G_k)$, for $i=1,2$.

  \begin{definition}\label{def:cigenerators}
    Consider three generators $G_1$, $G_2$, $G_k$. We call $G_1$ and $G_2$ {\em conditionally independent} generators given $G_k$ if there is no simultaneous move in both $G_1$ and $G_2$ without the coordinator $G_k$ being also involved. This condition can be written as
    \[
      E_r(G_1 \| G_2) \cap E_r(G_1) \cap E_r(G_2) \subseteq E_r(G_k)\,,
    \]
    where $E_r(G)$ denotes the set of all reachable symbols in $G$, see also \cite{KvS08}.
  \end{definition}

  The concept is easily extended to the case of three or more generators. The corresponding concept in terms of languages follows.

  \begin{definition}\label{def:cilanguages}
    Consider event sets $E_1$, $E_2$, $E_k$ and languages $L_1 \subseteq E_1^*$, $L_2 \subseteq E_2^*$, $L_k \subseteq E_k^*$. Languages $L_1$ and $L_2$ are said to be {\em conditionally independent} given $L_k$ if
    \[
      E_r(L_1 \| L_2 ) \cap E_1 \cap E_2 \subseteq E_k\,,
    \]
    where $E_r(L)$ is the set of all (reachable) symbols occurring in words of $L$.
  \end{definition}

  \begin{definition}
    A language $K$ is said to be {\em conditionally decomposable\/} with respect to event sets $(E_{1+k},E_{2+k},E_k)$ if
    \[
      K = P_{1+k} (K)\| P_{2+k} (K) \| P_k(K).
    \]
  \end{definition}

  It is not hard to prove that $K$ is conditionally decomposable if and only if there are languages $M_1\subseteq E_{1+k}^*$, $M_2\subseteq E_{2+k}^*$, $M_3\subseteq E_{k}^*$ such that $K=M_1\|M_2\|M_3$, see the following lemma.
  \begin{lem}\label{TCS}
    A language $M\subseteq E^*$ is conditionally decomposable with respect to event sets $(E_1,E_2,E_k)$ if and only if there exist languages $M_{i} \subseteq E_{i}^*$, $i=1,2,k$, such that $M = M_{1} \| M_{2}  \| M_k$.
  \end{lem}
  \begin{proof}
    Conditionally decomposability means that $M=P_1(M)\| P_2(M)\| P_k(M)$. Let $M_i=P_i(M)$, for $i=1,2,k$. Then the sufficiency is proven. To prove the necessity, assume that there are languages $M_{i} \subseteq E_{i}^*$, for $i=1,2,k$, such that $M = M_{1} \| M_{2}  \| M_k$. Obviously, $P_i(M)\subseteq M_i$, for $i=1,2,k$, which implies the inclusion $P_k(M)\|P_1(M)\|P_2(M)\subseteq M$. As it holds that $M\subseteq P_i^{-1}P_i(M)$, for $i=1,2,k$, the definition of synchronous product implies that $M\subseteq P_k(M)\|P_1(M)\|P_2(M)$. Hence, the lemma holds.
  \end{proof}

\section{Control synthesis of conditionally controllable languages} \label{sec:controlsynthesis}
  \begin{prob}\label{problem:controlsynthesis}
    Consider generators $G_1$, $G_2$, $G_k$ with event sets $E_1$, $E_2$, $E_k$, respectively, and a prefix-closed specification language
    \[
      K\subseteq L(G_1\| G_2 \| G_k)\,.
    \]

    We assume that $K$ is prefix-closed because we only focus on controllability issues in this paper, while nonblocking issues will be addressed in a future work.

    Assume that the coordinator $G_k$ makes the two generators $G_1$ and $G_2$ conditionally independent, and that the specification language $K$ is conditionally decomposable with respect to event sets $(E_{1+k},E_{2+k},E_k)$.

    The overall control task is divided into local subtasks and the coordinator subtask. The coordinator takes care of its ``part'' of the specification, namely $P_k(K)$. Otherwise stated, $S_k$ must be such that
    \[
      L(S_k/G_k)\subseteq P_k(K)\,.
    \]

    Similarly, supervisors $S_1$ and $S_2$ take care of their corresponding ``parts'' of the specification, namely $P_{i+k}(K)$, for $i=1,2$. Otherwise stated, $S_i$ must be such that
    \[
      L(S_i/ [G_i \| (S_k/G_k) ])\subseteq P_{i+k}(K)\,,
    \]
    for $i=1,2$.

    The aim is to determine supervisors $S_1$, $S_2$, and $S_k$ for the respective generators so that the closed-loop system with the coordinator is such that
    \begin{gather*}
       L(S_1/ [G_1 \| (S_k/G_k) ]) \| L(S_2/ [G_2 \| (S_k/G_k) ]) \| L(S_k/G_k)
        = K\,.\qedd
    \end{gather*}
  \end{prob}

  \begin{definition}\label{def:conditionalcontrollability}
    Consider the setting of Problem \ref{problem:controlsynthesis}. We call the specification language $K \subseteq E^*$ {\em conditionally controllable} for generators $(G_1,G_2,G_k)$ and locally uncontrollable event sets $(E_{1+k,u},E_{2+k,u},E_{k,u})$ if
    \begin{enumerate}
      \item[(i)]
        $P_k(K) \subseteq E_k^*$ is controllable with respect to $L(G_k)$ and $E_{k,u}$; equivalently,
        \begin{align*}
          P_k(K) E_{k,u} \cap L(G_k) \subseteq P_k(K)\,.
        \end{align*}

      \item[(ii.a)]
        the language $P_{1+k}(K) \subseteq (E_1 \cup E_k)^*$ is controllable with respect to the language $L(G_1) \| P_k(K) \| P_k^{2+k} (L(G_2) \| P_k(K))$ and $E_{1+k,u}$; equivalently,
        \begin{align*}
          P_{1+k}(K) E_{1+k,u} \cap L(G_1) \| P_k(K) \| P_k^{2+k} (L(G_2) \| P_k(K)) \subseteq P_{1+k}(K)\,.
        \end{align*}

      \item[(ii.b)]
        the language $P_{2+k}(K) \subseteq (E_2 \cup E_k)^*$ is controllable with respect to the language $L(G_2) \| P_k(K) \| P_k^{1+k} (L(G_1) \| P_k(K))$ and $E_{2+k,u}$; equivalently,
        \begin{align*}
          P_{2+k}(K) E_{2+k,u} \cap L(G_2) \| P_k(K) \| P_k^{1+k} (L(G_1) \| P_k(K)) \subseteq P_{2+k}(K)\,.
        \end{align*}
    \end{enumerate}
  \end{definition}

  The interpretation of the term after the intersection in (ii.a) is that the effect of the subsystem $G_1$ in combination with the controlled coordinator system $G_2 \| P_k(K)$ has to be taken into account when checking conditional controllability.

  Since $P_k(K)$ is controllable with respect to $L(G_k)$ and $E_{k,u}$, there exists a supervisor $S_k$ such that
  \[
    P_k(K) = L(S_k/G_k)\,.
  \]

  Note that the conditions of Definition \ref{def:conditionalcontrollability} can be checked by classical algorithms with low (polynomial) computational complexity for verification of controllability as is directly clear from the definition.

  However, the complexity of checking conditional controllability is much less than that for the global system $G_1 \| G_2 \|G_k$. This is because instead of checking controllability with the global specification and the global system, we check it only on the corresponding projections to $E_{1+k}$ and $E_{2+k}$. The projections are smaller when they satisfy the observer property (see Definition~\ref{def5} below).

\subsection{Auxiliary results}
  In this section, we present several auxiliary results that are useful in the rest of this paper. First, let us recall the following result proven in \cite{Won04} showing when a natural projection can be distributed over a synchronous product.

  \begin{lem}\label{lemma:Wonham}
    Let $E_k\subseteq E=E_1\cup E_2$ be event sets such that $E_1\cap E_2\subseteq E_k$. Let $L_1\subseteq E_1^*$ and $L_2\subseteq E_2^*$ be local languages. Let $P_k : E^*\to E_k^*$ be a natural projection, then
    \[
      P_k(L_1\| L_2)=P_{1\cap k}^1 (L_1) \| P_{2\cap k}^2 (L_2)\,,
    \]
    where $P^i_{i\cap k} : E_i^* \to E_k^*$, for $i=1,2$.
  \end{lem}

  An immediate consequence of Lemma~\ref{lemma:Wonham} and the definition of synchronous product is the following lemma proven in~\cite{FLT}.
  \begin{lem}\label{lemF}
    Let $E_k\subseteq E=E_1\cup E_2$ be event sets such that $E_k = E_1\cap E_2$. Let $L_1\subseteq E_1^*$ and $L_2\subseteq E_2^*$ be local languages. Let $P_i : E^*\to E_i^*$ and $P^j_k:E_j^*\to E_k^*$ be natural projections, for $i=1,2,k$ and $j=1,2$. Then, for $\{i,j\}=\{1,2\}$,
    \[
      P_i(L_1 \| L_2)= L_i \cap (P^i_k)^{-1}P^j_k(L_j)\,.
    \]
  \end{lem}

  \begin{lem}\label{simple}
    Let $L\subseteq E^*$ be a language and $P_k: E^*\to E_k^*$ be a natural projection with $E_k\subseteq E$, for some event set $E$. Then
    \[
      L\| P_k(L)=L\,.
    \]
  \end{lem}
  \begin{proof}
    As $L\subseteq P_k^{-1}P_k(L)$, we obtain by the definition of the synchronous product that $L\| P_k(L)=L\cap P_k^{-1}P_k(L)=L$.
  \end{proof}

\subsection{Control synthesis of conditionally controllable\\ languages}
  The following theorem presents the necessary and sufficient condition on a specification language to be exactly achieved in the coordination control architecture.

  \begin{theorem}\label{th:controlsynthesissafety}
    Consider the setting of Problem \ref{problem:controlsynthesis}. There exist supervisors $S_1$, $S_2$, $S_k$ such that
    \begin{equation}\label{eq:controlsynthesissafety}
      \begin{aligned}
        L(S_1/[G_1 \| (S_k/G_k)]) ~ & \| ~ L(S_2/[G_2 \| (S_k/G_k)])
        ~ \| ~ L(S_k/G_k)) ~ = ~ K
      \end{aligned}
    \end{equation}
    if and only if the specification $K$ is conditionally controllable for generators $(G_1,G_2,G_k)$ and locally uncontrollable event sets $(E_{1+k,u},E_{2+k,u},E_{k,u})$.
  \end{theorem}
  \begin{proof}
    To prove the sufficiency, let the specification language $K$ be conditionally controllable for generators $(G_1,G_2,G_k)$ and locally uncontrollable event sets $(E_{1+k,u},E_{2+k,u},E_{k,u})$. It must be checked that (\ref{eq:controlsynthesissafety}) holds.

    However, as
    \[
      K\subseteq L(G_1\|G_2\|G_k) \Rightarrow P_k(K) \subseteq L(G_k)\,,
    \]
    and $P_k(K)$ is controllable with respect to $L(G_k)$ and $E_{k,u}$, it follows from \cite{RW87} that there exists a supervisor $S_k$ over the event set $E_k$ such that
    \[
      L(S_k/G_k) = P_k(K)\,.
    \]

    Next, consider the language
    \begin{multline*}
      L(G_1) ~ \| ~ L(S_k/G_k) ~ \cap ~ (P_k^{1+k})^{-1} P_k^{2+k} L(G_2 \| (S_k/G_k))\\
       = L(G_1) ~ \| ~ L(S_k/G_k) ~ \| ~ P_k^{2+k} L(G_2 \| (S_k/G_k))\,,
    \end{multline*}
    by the definition of the synchronous product. Furthermore,
    \begin{align*}
      K & \subseteq L(G_1\| G_2\| G_k) \\
        & \Rightarrow \\
      P_{1+k}(K) & \subseteq  P_{1+k}L(G_1\| G_2\| G_k) \\
        & = P_{1+k}(L(G_1)\|L(G_k)) ~ \| ~ P_{k \cap 2}^2 L(G_2)\,, ~ ~ ~ ~ \text{by Lemma~\ref{lemma:Wonham},}\\
        & = L(G_1) \| L(G_k) \| P_{k \cap 2}^2 L(G_2)\,.
    \end{align*}
    Then,
    \begin{align*}
      P_{1+k}(K) & \subseteq L(G_1) \|  P_{k \cap 2}^2 L(G_2) \|  L(G_k)
      \text{ and}\\
      P_{1+k}(K) & \subseteq  (P_{k}^{1+k})^{-1} P_k(K) \\
      & \Rightarrow  \\
        P_{1+k}(K)
      & \subseteq  L(G_1) \|  P_{k \cap 2}^2 L(G_2) \|  L(G_k) \|  P_k(K) \\
      & =  L(G_1) \|  P_{k \cap 2}^2 L(G_2) \|  L(G_k) \|  L(S_k/G_k) \\
      & =  L(G_1) \|  P_{k \cap 2}^2 L(G_2) \|  L(S_k/G_k)\,,\\
        & \hspace{5.4cm}\text{by $L(G_k)\|L(S_k/G_k) = L(S_k/G_k)\,,$}\\
      & =  L(G_1) \|  P_{k \cap 2}^2 L(G_2) \|  L(S_k/G_k) \| P_k(K) \\
      & =  L(G_1) \|  L(S_k/G_k) \|  P_{k}^{2+k} L(G_2 \| S_k/G_k))\,,\\
        & \hspace{3,5cm} \text{by $P_{k \cap 2}^2 L(G_2) \| P_k(K) = P_k^{2+k} L(G_2 \| (S_k/G_k))\,.$}
    \end{align*}

    From the above relations and the assumption that the system is conditionally controllable then follows that there exists a supervisor $S_1$ such that
    \[
      L(S_1/ [ G_1 \|  (S_k/G_k) \|  P_k^{2+k} (G_2 \|  (S_k/G_k))])
      = P_{1+k}(K)\,.
    \]
    Because of Condition (ii.b) of Definition \ref{def:conditionalcontrollability}, a similar argument shows that there exists a supervisor $S_2$ such that
    \[
      L(S_2/ [ G_2 \| (S_k/G_k) \| P_k^{1+k} (G_1 \| (S_k/G_k))])
      = P_{2+k}(K)\,.
    \]
    In addition,
    \begin{multline}\label{lab_k}
      L(S_i/ [ G_i \| (S_k/G_k) \| P_k^{i+k} (G_i \| (S_k/G_k))])\\
      \begin{aligned}
        & = L(S_i) \| L(G_i \| (S_k/G_k)) \| P_k^{i+k} L(G_i \| (S_k/G_k))\\
        & = L(S_i) \| L(G_i \| (S_k/G_k)), ~ ~ ~ ~ ~ ~ ~ ~ ~ ~\text{by Lemma~\ref{simple},}\\
        & = L(S_i/[G_i \| (S_k/G_k)])\,,
      \end{aligned}
    \end{multline}
    which follows from the properties of the synchronous product. It is now sufficient to notice that
    \begin{multline*}
         L(S_1/ [ G_1 \| (S_k/G_k) \| P_k^{2+k} (G_2 \| (S_k/G_k))])\\
     \| ~
        L(S_2/ [ G_2 \| (S_k/G_k) \| P_k^{1+k} (G_1 \| (S_k/G_k))])
    \end{multline*}
    can be rewritten using the commutativity of the synchronous product exchanging the third and the last component as
    \begin{multline*}
      L(S_1) ~ \| ~ L(G_1 \| (S_k/G_k)) ~ \| ~ P_k^{1+k} L(G_1 \| (S_k/G_k))\\
      ~ \| ~ L(S_2) ~ \| ~ L(G_2 \| (S_k/G_k)) ~ \| ~ P_k^{2+k} L(G_2 \| (S_k/G_k))
    \end{multline*}
    which is reduced, using $(\ref{lab_k})$, to
    \begin{multline*}
      L(S_1) ~ \| ~ L(G_1 \| (S_k/G_k)) ~ \| ~ L(S_2) ~ \| ~ L(G_2 \| (S_k/G_k))\\
      = L(S_1/[G_1 \| (S_k/G_k)]) ~ \| ~ L(S_2/[G_2 \| (S_k/G_k)])\,.
    \end{multline*}
    Finally, since $K$ is conditionally decomposable and equalities
    \begin{align*}
      P_{1+k}(K) & =  L(S_1/ [ G_1 \| (S_k/G_k) \| P_k^{2+k} (G_2 \| (S_k/G_k))])\\
      P_{2+k}(K) & =  L(S_2/ [ G_2 \| (S_k/G_k) \| P_k^{1+k} (G_1 \| (S_k/G_k))])\\
      P_k(K) & =  L(S_k/G_k)
    \end{align*}
    are proven above, it follows that
    \begin{gather*}
      L(S_1/[G_1 \| (S_k/G_k)]) \| L(S_2/[G_2 \| (S_k/G_k)]) \| L(S_k/G_k))\\
        \begin{aligned}
          = &~ P_{1+k}(K) \| P_{2+k}(K) \| P_k(K)
          = K\,.
        \end{aligned}
    \end{gather*}
    Thus, the sufficiency is proven.

    To prove the necessity, projections $P_k$, $P_{1+k}$, $P_{2+k}$ will be applied to Equality~(\ref{eq:controlsynthesissafety}). Let us recall that since all the supervisors are admissible, the closed-loop languages can be written as corresponding synchronous products. This means that (\ref{eq:controlsynthesissafety}) can be rewritten as follows.
    \begin{align*}
      K & = L(S_1) \| L(G_1) \| L(S_k) \| L(G_k) \| L(S_2) \| L(G_2) \| L(S_k) \| L(G_k) \| L(S_k) \| L(G_k)\\
        & = L(S_1) \| L(G_1) \| L(S_2) \| L(G_2) \| L(S_k) \| L(G_k)\,,
    \end{align*}
    which yields after projecting by $P_k$
    \begin{align*}
      P_k(K)  & = P_k(L(S_1) \| L(G_1) \| L(S_2) \| L(G_2) \| L(S_k) \| L(G_k))\\
              & = L(S_k) \| L(G_k) \cap P_k(L(S_1) \| L(G_1) \| L(S_2) \| L(G_2))\\
              & \subseteq L(S_k) \| L(G_k)\\
              & = L(S_k/G_k)\,.
    \end{align*}

    On the other hand, we always have  $L(S_k/G_k) \subseteq P_k(K)$ because $S_k$ is a supervisor that enforces the coordinator part of the specification $P_k(K)$. Hence, we have that
    \[
      L(S_k/G_k) = P_k(K)\,,
    \]
    which means according to the basic controllability theorem of supervisory control that $P_k(K) \subseteq E_k^*$ is controllable with respect to $L(G_k)$ and $E_{k,u}$, i.e., (i) of the definition of conditional controllability is satisfied.

    Now, (ii.a) of conditional controllability will be shown; (ii.b) is a symmetric condition. The application of $P_{1+k}$ to (\ref{eq:controlsynthesissafety}) yields
    \begin{align*}
      P_{1+k} \Bigl( L(S_k/G_k) & \| L(S_1/ [G_1 \| (S_k/G_k)]) \| L(S_2/ [ G_2 \| (S_k/G_k)]) \Bigr) = P_{1+k}(K)\,.
    \end{align*}
    Since $E_{1+k}\cap E_{2+k}=E_k$, $L(S_2)\| L(G_2 \| (S_k/G_k))=L(S_2)\cap L(G_2 \| (S_k/G_k))$ because both components are over the same event set $E_{2+k}$, and the fact that $P_{1+k}^{2+k}=P_k^{2+k}$ imply that
    \begin{align*}
      P_{1+k}(K)  & = L(S_k/G_k)\, \|\, L(S_1/ [ G_1 \| (S_k/G_k)])\, \|\, P_{k}^{2+k} L(S_2/ [ G_2 \| (S_k/G_k)])\\
                  & = L(S_k/G_k)\, \|\, L(S_1/ [ G_1 \| (S_k/G_k)])\, \|\, P_{k}^{2+k} (L(S_2)\| L(G_2 \| (S_k/G_k)))\\
                  & \subseteq L(S_k/G_k)\, \|\, L(S_1/ [ G_1 \| (S_k/G_k)])\, \|\, P_{k}^{2+k} L(G_2 \| (S_k/G_k))\\
                  & \subseteq L(S_1/ [ G_1 \| (S_k/G_k)])\, \|\, P_{k}^{2+k} L(G_2 \| (S_k/G_k))\\
                  & \subseteq L(S_1/ [ G_1 \| (S_k/G_k)])\\
                  & \subseteq P_{1+k}(K)\,.
    \end{align*}

    Using again the fact that the closed-loop behavior under admissible supervisors can be recast as a synchronous composition of the plant and the supervisor, we finally get
    \begin{gather*}
        L(S_1) ~ \| ~ L(G_1) \| L(S_k/G_k) \| P_{k}^{2+k} L(G_2 \| (S_k/G_k)) 
      = P_{1+k}(K)\,.
    \end{gather*}

    In this equality, the whole term $G_1 \| (S_k/G_k) \| P_{k}^{2+k} (G_2 \| (S_k/G_k))$ after $L(S_1)$ can be taken as a new plant. According to the basic controllability theorem of supervisory control this equality implies that $P_{1+k}(K)$ is controllable with respect to $L(G_1 \| (S_k/G_k) \| P_{k}^{2+k} (G_2 \| (S_k/G_k)))$ and $E_{1+k,u}$, i.e., (ii.a) of the definition of conditional controllability is satisfied, which was to be shown.
  \end{proof}

  The interest in Theorem \ref{th:controlsynthesissafety} is in the computational savings in the computation of supervisors. The distributed way of constructing successively the supervisors $S_1$, $S_2$, $S_k$ is much less complex than the construction of the global supervisor for the system $G_1 \| G_2 \| G_k$.

  Note that it is required that $L(S_k/G_k)\subseteq P_k(K)$. Similarly, it is required that $L(S_i/ [G_i \| (S_k/G_k) ])\subseteq P_{i+k}(K)$, for $i=1,2$. Otherwise stated, we are looking for necessary conditions on global specifications for having the maximal permissivity of the language resulting by the application of the control scheme only in the (reasonable) case where safety can be achieved by the supervisors $S_k,S_1,$ and $S_2$. We have proven that in such a case conditional controllability is necessary for the optimality (maximal permitting). It is clear from the proof that for the sufficiency part we need not assume the inclusions above (cf.~\cite{KvS08}).

  In practice it is more interesting to know when safety (i.e., inclusion) holds when applying the overall control scheme combining a coordinator with local supervisors.

  Similarly as in the monolithic case it may happen that the maximal acceptable behavior given by the specification language $K$ is not achievable using the coordination control scheme. It follows from Theorem~\ref{th:controlsynthesissafety} that in our case such a situation occurs whenever $K$ is not conditionally controllable. A natural question is to find the best approximation from below: a conditionally controllable sublanguage. It turns out that the following result holds true.

  \begin{theorem} \label{existence}
    The supremal conditionally controllable sublanguage of a given specification $K$ always exists and is equal to the union of all conditionally controllable sublanguages of $K$.
  \end{theorem}
  \begin{proof}
    Similarly as for ordinary controllability it can be shown that conditional controllability is preserved by language unions.
  \end{proof}

  \begin{exmp}
    Consider the following generators over the event sets
    \[
      E_k=\{a,b,e,\varphi\}\subseteq E_1\cup E_2 = \{a,d,e,\varphi\} \cup \{b,f,\varphi\}\,,
    \]
    where the set of controllable events is $E_c=\{e,b,\varphi\}$. Define
    \begin{itemize}
      \item $G_1=(\{1,2,3,4\},\{a,d,e,\varphi\},f_1,1,\{1\})$ with the transition function $f_1$ defined in Figure~\ref{fig3a4}\subref{fig3},
      \item $G_2=(\{1,2,3\},\{b,\varphi,f\},f_2,1,\{1\})$ with the transition function $f_2$ defined in Figure~\ref{fig3a4}\subref{fig4}, and
      \item the coordinator $G_k=(\{1,2,3\},$ $\{a,b,\varphi\},f_k,1,\{1\})$ with $f_k$ defined in Figure~\ref{fig3a4}\subref{fig5}.
    \end{itemize}
    \begin{figure}[ht]
      \centering
      \subfloat[Generator for $G_1$.]{\label{fig3}\includegraphics{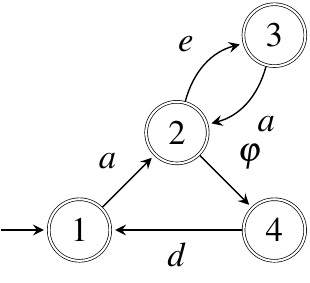}}
      \qquad\qquad
      \subfloat[Generator for $G_2$.]{\label{fig4}\includegraphics{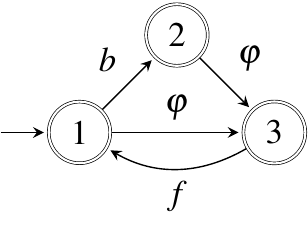}}
      \qquad\qquad
      \subfloat[Generator for $G_k$.]{\label{fig5}\includegraphics{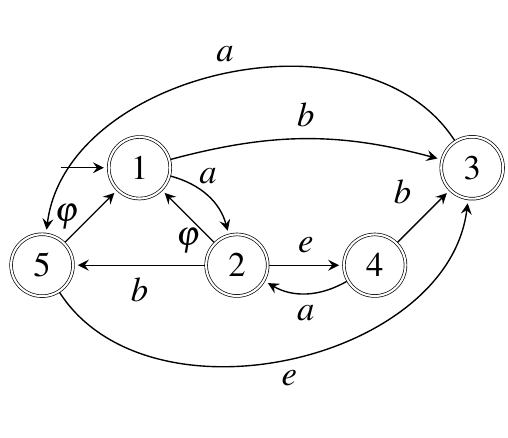}}
      \caption{Generators for $G_1$, $G_2$, and $G_k$.}
      \label{fig3a4}
    \end{figure}
    Assume the specification $K$ is described by the DES generator
    \[
      D=(\{1,2,3,4,5,6,7\},\{a,b,d,f,\varphi\},\delta,1,\{1\})\,,
    \]
    where $\delta$ is defined as in Figure~\ref{fig6}.
    \begin{figure}[ht]
      \begin{center}
        \includegraphics{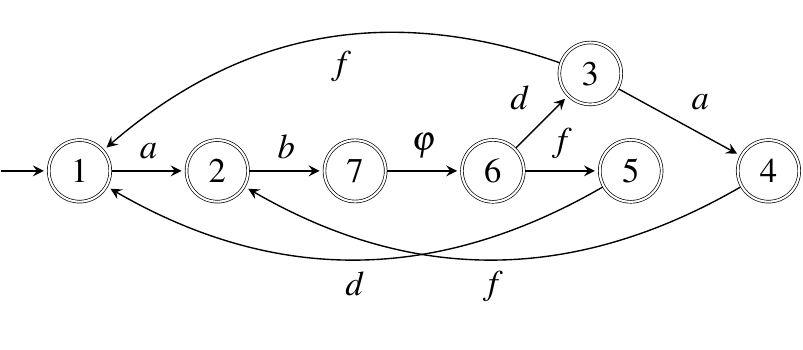}
        \caption{Generator for $D$.}
        \label{fig6}
      \end{center}
    \end{figure}

    It can be verified that $G_k$ makes $G_1$ and $G_2$ conditionally independent and that the specification language $K$ is conditionally decomposable. In addition, $P_k(K)$, $P_{1+k}(K)$, and $P_{2+k}(K)$ are controllable with respect to languages $L(G_k)$, $L(G_1)\| P_k(K) \| P^{2+k}_k (L(G_2) \| P_k(K))$, and $L(G_2)\| P_k(K) \| P^{1+k}_k (L(G_1) \| P_k(K))$, respectively. The automata representations of supervisors $S_1$, $S_2$, and $S_k$ coincide with generators $P_k(K)$, $P_{1+k}(K)$, and $P_{2+k}(K)$, respectively, see Figure~\ref{figPk}.
    Then, obviously,
    \begin{align*}
      L(S_1/[G_1 \| (S_k/G_k)]) ~ \| ~ L(S_2/[G_2 \| (S_k/G_k)]) ~ \| ~ L(S_k/G_k)) ~ = ~ K\,.
    \end{align*}
    \begin{figure}[ht]
      \centering
      \subfloat[Generator for $P_k(K)=S_k$.]{\label{fig1}\includegraphics[scale=.98]{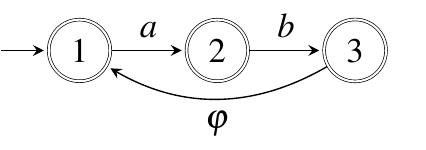}}
      \qquad\qquad
      \subfloat[Generator for $P_{1+k}(K)=S_1$.]{\label{fig2b}\includegraphics[scale=.98]{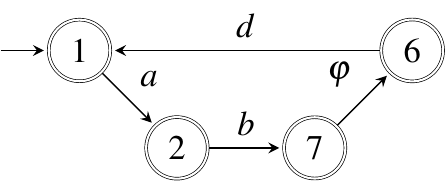}}
      \qquad\qquad
      \subfloat[Generator for $P_{2+k}(K)=S_2$.]{\label{fig2a}\includegraphics[scale=.98]{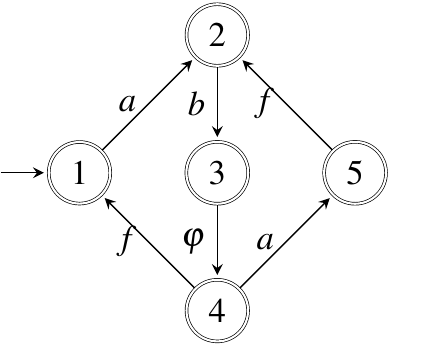}}
      \caption{Generators for supervisors $S_k$, $S_1$, and $S_2$.}
      \label{figPk}
    \end{figure}
  \end{exmp}

\section{Supremal conditionally controllable\\ sublanguages}\label{sec:procedure}
  In this section, we present a procedure for the computation of the supremal conditionally controllable sublanguage to a given specification language $K$. Assume generators $G_1$, $G_2$, and $G_k$ are given. In what follows, we use the notation $L_i=L(G_i)$, for $i=1,2,k$. Let $\supCC(K, L, (E_{1+k,u}, E_{2+k,u}, E_{k,u}))$ denote the supremal conditionally controllable sublanguage of $K$ with respect to $L=L(G_1\| G_2\| G_k)$ and uncontrollable event sets $(E_{1+k,u}, E_{2+k,u}, E_{k,u})$. This approach is based on concepts from hierarchical supervisory control, which is natural because the coordination control can be seen as a combination of decentralized and hierarchical supervisory control.

\subsection{Auxiliary results and definitions}
  First, additional results and definitions required in the rest of this paper are introduced. Several lemmas recall and deepen the knowledge concerning natural projections. Then, definitions of two important notions are recalled.

  \begin{lem}\label{lem9}
    Let $E=E_1\cup E_2$ and $E_k$ be event sets such that $E_1\cap E_2 \subseteq E_k$, and let $L_1\subseteq E_1^*$ and $L_2\subseteq E_2^*$ be two languages. Let $P_k:E^* \to E_k^*$ be a natural projection. Then,
    \[
      P_k(L_1\| L_2) = P^{1+k}_k(P^{1+k}_1)^{-1}(L_1)\cap P^{2+k}_k(P^{2+k}_2)^{-1}(L_2)\,.
    \]
  \end{lem}
  \begin{proof}
    This follows from Lemma~\ref{lemma:Wonham}, the definition of the synchronous product, and Proposition~4.2(6) in \cite{FLT} showing the commutativity
    \[
      (P^{k}_{i\cap k})^{-1} P^{i}_{i\cap k} = P^{i+k}_k(P^{i+k}_i)^{-1}\,,
    \]
    for $i=1,2$. Specifically, in turn we have
    \begin{align*}
      P_k(L_1\| L_2) & = P_{1\cap k}^1(L_1) \| P_{2\cap k}^2(L_2)\\
                     & = (P^{k}_{1\cap k})^{-1} P^{1}_{1\cap k} (L_1) \cap (P^{k}_{2\cap k})^{-1} P^{2}_{2\cap k} (L_2)\\
                     & = P^{1+k}_k(P^{1+k}_1)^{-1}(L_1)\cap P^{2+k}_k(P^{2+k}_2)^{-1}(L_2)\,,
    \end{align*}
    which proves the lemma.
  \end{proof}

  \begin{lem}\label{lem10}
    Let $E=E_1\cup E_2$ and $E_k$ be event sets such that $E_1\cap E_2 \subseteq E_k$, and let $L_1\subseteq E_1^*$, $L_2\subseteq E_2^*$, and $C_k\subseteq E_k^*$ be languages. Let $P_k^{i+k}:(E_i\cup E_k)^* \to E_k^*$ be a natural projection. Then,
    \[
      P^{i+k}_k(L_i \| C_k) = P^{i+k}_k(P^{i+k}_i)^{-1}(L_i) \cap C_k\,.
    \]
  \end{lem}
  \begin{proof}
    This follows from Lemma~\ref{lem9}.
  \end{proof}

  \begin{lem}\label{lem1}
    Let $E'\subseteq E$ be two event sets. Let $M\subseteq {E'}^*$ be a language, and let $P:E^*\to {E'}^*$ be a natural projection. Then $M$ is prefix-closed if and only if $P^{-1}(M)$ is prefix-closed.
  \end{lem}
  \begin{proof}
    Assume that $P^{-1}(M)$ is prefix-closed. Let $w\in M$, then $P(w)=w$ and, therefore, $w\in P^{-1}(M)$. For each prefix $s$ of $w$, $s\in P^{-1}(M)$. However, $P(s)=s\in M$. On the other hand, assume that $M$ is prefix-closed. Let $w\in P^{-1}(M)$ and $x$ be its prefix. Then $w=xy$, for some $y\in E^*$, and $P(w)=P(x)P(y)\in M$. Thus, $P(x)\in M$, which implies $x\in P^{-1}(M)$.
  \end{proof}

  The following lemma extending the definition of controllability is proven in \cite{brandt}.
  \begin{lem}\label{lem2}
    Let $K\subseteq L$ be two prefix-closed languages over an event set $E$. Then $K$ is controllable with respect to $L$ and $E_{u}$ if and only if \[
      KE_{u}^* \cap L \subseteq K\,.
    \]
  \end{lem}

  \begin{lem}\label{lem11}
    Let $E=E_1\cup E_2$ be event sets, and let $L_1\subseteq E_1^*$ and $L_2\subseteq E_2^*$ be two languages. Let $P_i : E^* \to E_i^*$ be natural projections, for $i=1,2$. Let $A\subseteq E^*$ be a language such that $P_1(A)\subseteq L_1$ and $P_2(A)\subseteq L_2$. Then
    \[
      A\subseteq L_1\|L_2\,.
    \]
  \end{lem}
  \begin{proof}
    As $A\subseteq P_i^{-1}P_i(A)$, for $i=1,2$, it follows that
    \begin{align*}
      A & \subseteq P_1^{-1}P_1(A) \cap P_2^{-1}P_2(A)\\
        & = P_1(A) \| P_2(A), \hspace{1.6cm}\text{by definition},\\
        & \subseteq L_1 \| L_2\,.
    \end{align*}
    Hence, the lemma holds true.
  \end{proof}

  \begin{lem}[Transitivity of controllability]\label{lem_trans}
    Let $K\subseteq L \subseteq M$ be languages over an event set $E$ such that $K$ is controllable with respect to $L$ and $E_u$, and $L$ is controllable with respect to $M$ and $E_u$. Then $K$ is controllable with respect to $M$ and $E_u$.
  \end{lem}
  \begin{proof}
    From the assumptions we know that
    \begin{align*}
      KE_u \cap L \subseteq K\quad \text{ and }\quad LE_u \cap M \subseteq L
    \end{align*}
    and we want to show that $KE_u \cap M \subseteq K$.

    Assume that $s\in K$, $a\in E_u$, and $sa\in M$. Then, $K\subseteq L$ implies that $s\in L$. As $sa\in M$, it follows from controllability of $L$ with respect to $M$ that $sa\in L$. However, $sa\in L$ implies that $sa\in K$, by controllability of $K$ with respect to $L$. Hence, the proof is complete.
  \end{proof}

  The following concepts \cite{WW96,FLT} are required in the main result of this section. These concepts are stemming from hierarchical supervisory control \cite{WW96}. It should not be surprising that they play a role in our study, because coordination control can be seen as a particular instance of hierarchical control.

  \begin{definition}\label{def5}
    The natural projection $P:E^* \to E_k^*$, where $E_k\subseteq E$ are event sets, is an {\em $L$-observer} for $L\subseteq E^*$ if, for all $t\in P(L)$ and $s\in \overline{L}$, if $P(s)$ is a prefix of $t$, then there exists $u\in E^*$ such that $su\in L$ and $P(su)=t$.
  \end{definition}

  \begin{definition}
    The natural projection $P:E^*\to E_k^*$, where $E_k\subseteq E$ are event sets, is {\em output control consistent} (OCC) for $L\subseteq E^*$ if for every $s\in \overline{L}$ of the form
    \[
      s=\sigma_1\sigma_2\dots\sigma_\ell \quad \textrm{ or }\quad  s = s'\sigma_0 \sigma_1 \dots \sigma_\ell, \,\quad \ell\ge 1\,,
    \]
    where $\sigma_0, \sigma_\ell\in E_k$ and $\sigma_i \in E\setminus E_k$, for $i=1,2,\dots,\ell-1$, if $\sigma_\ell \in E_{u}$, then $\sigma_i \in E_{u}$, for all $i=1,2,\dots,\ell-1$.
  \end{definition}

\subsection{Computation of supremal conditionally controllable sublanguages}
  Now, we can present the main result of this section, which gives a procedure for the computation of supremal conditionally controllable sublanguages.

  \begin{theorem}\label{thm2}
    Let $K$ and $L=L_1\|L_2\|L_k$ be two prefix-closed languages over an event set $E=E_{1}\cup E_{2}\cup E_k$, where $L_i\subseteq E_i^*$, for $i=1,2,k$, and let the specification language $K$ be conditionally decomposable. Define the languages
    \begin{align*}
      \supC_k     & = \supC(P_k(K) \| P_k(L_1\| L_2) \| L_k, L_k, E_{k,u})\,,\\
      \supC_{1+k} & = \supC(P_{1+k}(K) \| L_1, L_1 \| \supC_k, E_{1+k,u})\,,\\
      \supC_{2+k} & = \supC(P_{2+k}(K) \| L_2, L_2 \| \supC_k, E_{2+k,u})\,.
    \end{align*}
    Let the projection $P^{i+k}_k$ be an $(P^{i+k}_i)^{-1}(L_i)$-observer and OCC for the language $(P^{i+k}_i)^{-1}(L_i)$, for $i=1,2$. Then,
    \begin{align*}
      \supC_k \| \supC_{1+k} \| \supC_{2+k} = \supCC(K\cap L, L, (E_{1+k,u}, E_{2+k,u}, E_{k,u}))\,.
    \end{align*}
  \end{theorem}
  \begin{proof}
    First, let us define
    \begin{align*}
           M & :=\supC_k \| \supC_{1+k} \| \supC_{2+k}\\
             & \text{ and }\\
      \supcc & := \supcc(K\cap L, L, (E_{1+k,u}, E_{2+k,u}, E_{k,u}))\,.
    \end{align*}
    To prove the first inclusion, $M\subseteq \supcc$, we show that
    \begin{enumerate}
      \item $M \subseteq K\cap L$ and
      \item $M$ is conditionally controllable with respect to the language $L$ and uncontrollable event sets $(E_{1+k,u}, E_{2+k,u}, E_{k,u})$.
    \end{enumerate}

    \smallskip\noindent
    1) First, notice that
      \begin{align*}
        M = \supC_k \| \supC_{1+k} \| \supC_{2+k}
          &\subseteq P_k(K)\| L_k \| P_{1+k}(K)\| L_1 \| P_{2+k}(K)\| L_2\\
          &= \underbrace{P_k(K) \| P_{1+k}(K)\| P_{2+k}(K)}_{K} \| \underbrace{L_k \| L_1 \| L_2}_{L}\\
          &= K \cap L
      \end{align*}
      since $K$ is conditionally decomposable and $L=L_1\|L_2\|L_k$.

    \smallskip\noindent
    2) To prove that $M$ is conditionally controllable with respect to the language $L$ and $(E_{1+k,u}, E_{2+k,u}, E_{k,u})$, we need to show the following three properties of Definition~\ref{def:conditionalcontrollability}:
    \begin{itemize}
      \item[(I)]   $P_k(M) E_{k,u} \cap L_k \subseteq P_k(M)$,
      \item[(II)]  $P_{1+k}(M)E_{1+k,u} \cap L_1 \| P_k(M) \| P^{2+k}_k(L_2 \| P_k(M))
                    \subseteq P_{1+k}(M)$,
      \item[(III)] $P_{2+k}(M)E_{2+k,u} \cap L_2 \| P_k(M) \| P^{1+k}_k(L_1 \| P_k(M))
                    \subseteq P_{2+k}(M)$.
    \end{itemize}
    As the last two properties are similar, we prove only (II).

    \smallskip\noindent
    (I) To prove that $P_k(M) E_{k,u} \cap L_k \subseteq P_k(M)$ note that
      \begin{align*}
        P_k(M) & = \supC_k
          \cap P^{1+k}_k(\supC_{1+k})
          \cap P^{2+k}_k(\supC_{2+k})\,,
      \end{align*}
      which follows from Lemma~\ref{lemma:Wonham} by replacing the synchronous product with the intersection (which can be done because the components are over the same event set).

      Let $x\in P_k(M)$, then there exists $w\in M$ such that $P_k(w)=x$. Assume that $a\in E_{k,u}$ is such that $xa\in L_k$. We need to show that
      \[
        xa \in P_k(M)\,.
      \]
      As $x\in P_k(M)\subseteq \supC_k$, it follows from controllability of $\supc_k$ with respect to $L_k$ and $E_{k,u}$ that
      \begin{align}\label{lab00}
        xa \in \supC_k\,.
      \end{align}
      Thus, it remains to show that
      \begin{align}\label{lab_star}
        xa\in P^{i+k}_k(\supc_{i+k})\,,
      \end{align}
      for $i=1,2$. To this end, note first that from the properties of natural projections we have that
      \begin{align}\label{lab0}
        P_{1+k}(w)\in P_{1+k}(M)\subseteq \supC_{1+k}\,,
      \end{align}
      and $a\in E_{k,u}\subseteq E_{1+k,u}$. Next, by the definition of the synchronous product we obtain that
      \begin{align}\label{lab1}
        L_1 \| \supC_k = (P^{1+k}_1)^{-1}(L_1)\cap (P^{1+k}_k)^{-1}(\supC_k)\,.
      \end{align}
      Furthermore, $P_k^{1+k}(P_{1+k}(w)a) = xa \in \supC_k$, which implies that $P_{1+k}(w)a\in (P^{1+k}_k)^{-1}(\supC_k)$. This and the fact that
      \begin{align}\label{lab2}
        \supC_k \subseteq P_k(K) & \| P_k(L_1\| L_2)\nonumber\\
                       =  P_k(K) & \cap P_k(L_1\| L_2)\nonumber\\
                       =  P_k(K) & \cap P^{1+k}_k(P^{1+k}_1)^{-1}(L_1)\\
                                 & \cap P^{2+k}_k(P^{2+k}_2)^{-1}(L_2)\,, ~~~\text{by Lemma~\ref{lem9}}\,,\nonumber
      \end{align}
      implies that
      \begin{align}\label{lab3}
        P_k^{1+k}(P_{1+k}(w)a) & \in P^{1+k}_k(P^{1+k}_1)^{-1}(L_1)\,.
      \end{align}
      In addition, it follows from $(\ref{lab0})$ and the definition of $\supc_{1+k}$ that
      \begin{align}\label{lab4}
        P_{1+k}(w)\in (P^{1+k}_1)^{-1}(L_1)\,.
      \end{align}
      As $P_k^{1+k}(P_{1+k}(w))$ is obviously a prefix of $P_k^{1+k}(P_{1+k}(w)a)$, and $P^{1+k}_k$ is an $(P^{1+k}_1)^{-1}(L_1)$-observer, we obtain that there exists $u\in E_{1+k}^*$ such that
      \begin{align}\label{lab4_5}
        P_{1+k}(w)ua\in (P^{1+k}_1)^{-1}(L_1)
      \end{align}
      and $P^{1+k}_k(P_{1+k}(w)ua) = P^{1+k}_k(P_{1+k}(w)a)$, which means that $u \in (E_1\setminus E_k)^*$. Since the language $L_1$ is prefix-closed, so is by Lemma~\ref{lem1} $(P^{1+k}_1)^{-1}(L_1)$. Therefore, $P_{1+k}(w)u\in (P^{1+k}_k)^{-1}(L_1)$. Note that $P^{1+k}_k(P_{1+k}(w)u) = x \in \supc_k$, i.e., $P_{1+k}(w)u \in (P^{1+k}_k)^{-1}(\supc_k)$. By $(\ref{lab1})$ we thus obtain that
      \begin{align}\label{lab5}
        P_{1+k}(w)u \in L_1 \| \supC_k\,.
      \end{align}
      As the natural projection $P^{1+k}_k$ is also OCC for $(P^{1+k}_1)^{-1}(L_1)$ and $P_{1+k}(w)ua$ satisfies that $a\in E_k$, $u\in (E_1\setminus E_k)^*$, and $a\in E_{u}$, it follows that
      \[
        u\in E_{u}^*\,.
      \]
      As $P_{1+k}(w)\in \supC_{1+k}$, $\supc_{1+k}$ is controllable with respect to $L_1 \| \supC_k$ and $E_{1+k,u}$, and $P_{1+k}(w)u \in L_1 \| \supC_k$, Lemma~\ref{lem2} (extended controllability) implies that
      \begin{align}\label{lab6}
        P_{1+k}(w)u \in \supC_{1+k}\,.
      \end{align}
      Recall that $P_{1+k}(w)ua \in (P^{1+k}_1)^{-1}(L_1)$ is satisfied by $(\ref{lab4_5})$.

      As we also have $P^{1+k}_k(P_{1+k}(w)ua) = xa \in \supc_k$, by $(\ref{lab00})$, we obtain by $(\ref{lab1})$ that $P_{1+k}(w)ua \in L_1 \| \supC_k$, which implies by controllability of $\supc_{1+k}$ with respect to the language $L_1\|\supc_k$ and $E_{1+k,u}$ that $P_{1+k}(w)ua \in \supC_{1+k}$, i.e.,
      \begin{align*}
        xa & = P^{1+k}_k(P_{1+k}(w)ua)
             \in P^{1+k}_k(\supC_{1+k})\,.
      \end{align*}
      Analogously, we can prove that $xa\in P^{2+k}_k(\supC_{2+k})$, which proves $(\ref{lab_star})$. Thus,
      \[
        xa\in P_k(M)\,,
      \]
      which was to be shown.

    \smallskip\noindent
    (II) Now, we show the other property, namely
      \begin{align*}
        P_{1+k}(M)E_{1+k,u} \cap L_1 \| P_k(M) \| P^{2+k}_k(L_2 \| P_k(M)) \subseteq P_{1+k}(M)\,.
      \end{align*}
      First, note that by Lemma~\ref{lemF} and the definition of synchronous product we obtain that
      \begin{align*}
        P_{1+k}(M) = (P^{1+k}_k)^{-1}(\supC_k)\,\cap\, \supC_{1+k} \cap (P^{1+k}_k)^{-1}P^{2+k}_k(\supC_{2+k})\,.
      \end{align*}
      Assume that $x\in P_{1+k}(M)$. This is if and only if there exists $w\in M$ such that $P_{1+k}(w)=x$. Then $x\in \supC_{1+k}$. Let there exist $a\in E_{1+k,u}$ such that
      \begin{align}\label{cross}
        xa \in L_1 \| P_k(M) \| P^{2+k}_k(L_2 \| P_k(M))\,.
      \end{align}
      We need to show that
      \begin{align}\label{lab8}
        xa \in P_{1+k}(M)\,.
      \end{align}
      As $P_k(M)\subseteq \supc_k$, it follows that
      \begin{align} \label{lab9}
                   L_1 \| P_k(M) \| P^{2+k}_k(L_2 \| P_k(M)) \subseteq  L_1 \| \supC_k \| P^{2+k}_k(L_2 \| \supC_k)\,.
      \end{align}
      From controllability of $\supC_{1+k}$ with respect to $L_1 \| \supC_k$ and $E_{1+k,u}$, and because of the following inclusion $L_1 \| \supC_k \| P^{2+k}_k(L_2 \| \supC_k) \subseteq L_1 \| \supC_k$, we obtain that
      \begin{align}\label{lab10}
        xa\in \supC_{1+k}\,.
      \end{align}
      However, we also know that
      \[
        P_k(w)\in P_k(M)\subseteq \supC_k ~~ \mbox{(see above)}
      \]
      and
      \[
        P_{2+k}(w)\in P_{2+k}(M)\subseteq \supC_{2+k}\,.
      \]
      (A) On one hand, if $a\in E_1\setminus E_k$, then because $P^{1+k}_k(xa)=P_k(wa)=P_k(w)$, we obtain that $P^{1+k}_k(xa)\in \supC_k$, and because $P^{1+k}_k(xa)=P^{2+k}_kP_{2+k}(wa)=P^{2+k}_kP_{2+k}(w)$, we obtain that $P^{1+k}_k(xa)\in P^{2+k}_k(\supC_{2+k})$, Hence, for $a\in E_1\setminus E_k$ we have shown that
      \[
        xa\in P_{1+k}(M)\,,
      \]
      which was to be shown.

      \noindent
      (B) On the other hand, if $a\in E_1\cap E_k$, then
        \begin{align}\label{lab11}
          xa\in L_1\| P_k(M) \Rightarrow P^{1+k}_k(xa)\in P_k(M)\subseteq \supC_k\,.
        \end{align}
        Thus, $xa\in (P^{1+k}_k)^{-1}(\supC_k)$ is satisfied, and it remains to show that
        \begin{align}\label{lab12}
          xa \in (P^{1+k}_k)^{-1}P^{2+k}_k(\supC_{2+k})\,.
        \end{align}
        However, from (\ref{cross}) and Lemma~\ref{lem10} it follows that
        \begin{align}\label{lab13}
          P^{1+k}_k(xa)\in P^{2+k}_k(P^{2+k}_2)^{-1}(L_2)\cap P_k(M)\,.
        \end{align}
        In addition, we have from the definition of $\supc_{2+k}$ that
        \begin{align}\label{lab14}
          P_{2+k}(w)\in (P^{2+k}_2)^{-1}(L_2)\,.
        \end{align}
        As $P_k^{2+k}(P_{2+k}(w))$ is obviously a prefix of $P_k^{2+k}(P_{2+k}(w)a)$, $P^{2+k}_k(P_{2+k}(w)a) = P^{1+k}_k(x)a \in P_k(M) \subseteq \supc_k \subseteq P^{2+k}_k(P^{2+k}_2)^{-1}(L_2)$, and the projection $P^{2+k}_k$ is an $(P^{2+k}_2)^{-1}(L_2)$-observer, there is $u\in E_{2+k}^*$ such that
        \begin{align}\label{lab15}
          P_{2+k}(w)ua\in (P^{2+k}_2)^{-1}(L_2)
        \end{align}
        with $P^{2+k}_k(P_{2+k}(w)ua) = P^{2+k}_k(P_{2+k}(w)a)$, i.e., $u \in (E_2\setminus E_k)^*$. Since the language $L_2$ is prefix-closed, so is by Lemma~\ref{lem1} the language $(P^{1+k}_2)^{-1}(L_2)$. Therefore, $P_{2+k}(w)u\in (P^{2+k}_2)^{-1}(L_2)$ is satisfied. Furthermore, note that $P^{2+k}_k(P_{2+k}(w)u) = P^{1+k}_k(x) \in P_k(M) \subseteq \supc_k$ means that $P_{2+k}(w)u \in (P^{2+k}_k)^{-1}(\supc_k)$. Together, we have by the definition of synchronous product that
        \begin{align}\label{lab16}
          P_{2+k}(w)u \in L_2 \| \supC_k\,.
        \end{align}
        As the projection $P^{2+k}_k$ is also OCC for $(P^{2+k}_2)^{-1}(L_2)$, and $P_{2+k}(w)ua$ satisfies that $a\in E_k$, $u\in (E_2\setminus E_k)^*$, and $a\in E_{u}$, it follows that
        \[
          u\in E_{u}^*\,.
        \]
        Since $P_{2+k}(w)\in \supC_{2+k}$, $\supc_{2+k}$ is controllable with respect to $L_2 \| \supC_k$ and $E_{2+k,u}$, and $P_{2+k}(w)u \in L_2 \| \supC_k$ is satisfied, Lemma~\ref{lem2} implies that
        \begin{align}\label{lab17}
          P_{2+k}(w)u \in \supC_{2+k}\,.
        \end{align}
        Finally, since $P^{2+k}_k(P_{2+k}(w)ua) = P^{1+k}_k(x)a \in P_k(M) \subseteq \supc_k$ by $(\ref{lab13})$, it follows by this, $(\ref{lab15})$, and the definition of synchronous product that $P_{2+k}(w)ua \in L_2 \| \supC_k$. From this and controllability of $\supc_{2+k}$ with respect to $L_2 \| \supc_k$ and $E_{2+k,u}$, it follows that $P_{2+k}(w)ua \in \supC_{2+k}$, i.e.,
        \begin{align*}
          P^{1+k}_k(x)a & = P^{2+k}_k(P_{2+k}(w)ua) \in P^{2+k}_k(\supC_{2+k})\,,
        \end{align*}
        which proves $(\ref{lab12})$. Thus,
        \[
          xa\in P_{1+k}(M)
        \]
        which was to be shown.

      \smallskip\noindent
      (III) The case $P_{2+k}(M) E_{2+k,u} \cap L_2 \| P_k(M) \| P^{1+k}_k(L_1 \| P_k(M)) \subseteq P_{2+k}(M)$ is proven analogously to the previous one.

    Hence, we have shown that $M$ is conditionally controllable with respect to $L=L_1 \| L_2 \| L_k$ and $(E_{1+k,u}, E_{2+k,u}, E_{k,u})$ and, thus,
    \[
      M\subseteq \supcc\,.
    \]

    \medskip
    To prove the opposite inclusion, $\supcc \subseteq M$, by Lemma~\ref{lem11} it is sufficient to show that
    \begin{itemize}
      \item $P_k(\supCC)\subseteq \supC_k$ and
      \item $P_{i+k}(\supCC)\subseteq \supC_{i+k}$, for $i=1,2$.
    \end{itemize}
    To prove this note that $P_k(\supCC)  \subseteq P_k(L) = P_k(L_1\| L_2) \cap L_k$, where the last equality is by using Lemma~\ref{lemF}, and that also $P_k(\supCC)\subseteq P_k(K)$. Thus, we have
    \begin{align*}
      P_k(\supCC) \subseteq P_k(K)\cap L_k \cap P_k(L_1\| L_2) = P_k(K)\| L_k \| P_k(L_1\| L_2)\,.
    \end{align*}
    As, in addition, $P_k(\supCC)$ is controllable with respect to $L_k$ and $E_{k,u}$,
    \[
      P_k(\supCC)\subseteq \supC_k
    \]
    is satisfied. Further, $P_{1+k}(\supCC) \subseteq P_{1+k}(K)$ and $P_{1+k}(\supCC) \subseteq P_{1+k}(L) \subseteq L_1\| L_k$, which implies that
    \[
      P_{1+k}(\supCC)\subseteq P_{1+k}(K) \| L_1\,.
    \]
    We know that the language $P_{1+k}(\supCC)$ is controllable with respect to the language $L_1 \| P_k(\supcc) \| P^{2+k}_k(L_2\| P_k(\supcc))$ and $E_{1+k,u}$. Recall that by $(\ref{lab2})$
    \[
      P_k(\supcc)\subseteq \supC_k \subseteq P^{2+k}_k(P^{2+k}_2)^{-1}(L_2)\,.
    \]
    Next, the following holds:
    \begin{multline*}
      L_1 \| P_k(\supcc) \| P^{2+k}_k(L_2\| P_k(\supcc)) \\
      \begin{aligned}
        & = L_1 \| P_k(\supcc) \| P_k(\supcc) \cap P^{2+k}_k(P^{2+k}_2)^{-1}(L_2) \\
        & = L_1 \| P_k(\supcc) \| P_k(\supcc)\\
        & = L_1 \| P_k(\supcc)\,.
      \end{aligned}
    \end{multline*}
    Since $P_k(\supcc)$ is controllable with respect to $L_k$ and $E_{k,u}$, it is also controllable with respect to $\supc_k\subseteq L_k$ and $E_{k,u}$. As $P_{1+k}(\supcc)$ is controllable with respect to $L_1\| P_k(\supcc)$ and $E_{1+k,u}$, and $L_1\|P_k(\supcc)$ is controllable with respect to $L_1\| \supc_k$ and $E_{1+k,u}$ by Proposition~4.6 in \cite{FLT} (since all the languages under consideration are prefix-closed), it follows by Lemma~\ref{lem_trans} that $P_{1+k}(\supCC)$ is controllable with respect to $L_1 \| \supC_k$ and $E_{1+k,u}$, which implies that
    \[
      P_{1+k}(\supCC)\subseteq \supC_{1+k}\,.
    \]
    The case of the property (ii.b) is proven analogously. Hence, we have proven that
    \[
      \supcc \subseteq M
    \]
    and the proof is complete.
  \end{proof}

  Note that if we know that the specification language $K$ is included in the global language $L$, the computation can be simplified as shown in the following corollary.
  \begin{cor}\label{cor1}
    Let $K\subseteq L=L_1\|L_2\|L_k$ be two prefix-closed languages over an event set $E=E_1\cup E_2\cup E_k$, where $L_i\subseteq E_i^*$, for $i=1,2,k$, and let $K$ be conditionally decomposable. Define the languages
    \begin{align*}
      \supC_k     & = \supC(P_k(K),L_k,E_{k,u})\,,\\
      \supC_{1+k} & = \supC(P_{1+k}(K), L_1 \| \supC_k,E_{1+k,u})\,,\\
      \supC_{2+k} & = \supC(P_{2+k}(K), L_2 \| \supC_k,E_{2+k,u})\,.
    \end{align*}
    Let the natural projection $P^{i+k}_k$ be an $(P^{i+k}_i)^{-1}(L_i)$-observer and OCC for the language $(P^{i+k}_i)^{-1}(L_i)$, for $i=1,2$.
    Then
    \begin{align*}
      \supC_k \| \supC_{1+k} \| \supC_{2+k}
      = \supCC(K, L, (E_{k,u}, E_{1+k,u}, E_{2+k,u}))\,.
    \end{align*}
  \end{cor}
  \begin{proof}
    If $K\subseteq L$, then
    \begin{align*}
      P_k(K) & \subseteq P_k(L)\\
             & =         P_k(L_1\| L_2 \| L_k)\\
             & =         P_k(L_1 \| L_2) \| L_k, ~~ \mbox{by Lemma}~\ref{lemma:Wonham}\,.
    \end{align*}
    From $L_1 \| L_2 \| L_k = P_1^{-1}(L_1)\cap P_2^{-1}(L_2)\cap P_k^{-1}(L_k)$
    we also have that
    \begin{align*}
      P_{i+k}(K) & \subseteq P_{i+k}(P_i^{-1}(L_i))
                  = (P^{i+k}_i)^{-1}(L_i)\,,
    \end{align*}
    for $i=1,2$.
    Since $P_k(K)\subseteq P_k(L_1 \| L_2) \| L_k$ and
    $P_{i+k}(K)\subseteq (P^{i+k}_i)^{-1}(L_i)$, for $i=1,2$,
    the proof then follows from the previous theorem.
  \end{proof}

  In addition to the procedure for computation of $\supcc$ in a distributed way, another consequence of the theorem above is interesting. Namely, under the conditions of Theorem \ref{thm2}, $\supcc$ is conditionally decomposable (cf. Lemma~\ref{TCS}).

  Even more, the supremal conditionally controllable sublanguage is controllable with respect to the global plant as we show below and, consequently, the supremal conditionally controllable sublanguage is included in the global supremal controllable sublanguage. This is not a surprise because the language synthesized using the coordination architecture is more restrictive than the language synthesized using (monolithic) supervisory control of the global plant.

  \begin{theorem}\label{thm3}
    In the setting of Corollary \ref{cor1} we have that
    \[
      \supCC(K, L, (E_{k,u}, E_{1+k,u}, E_{2+k,u}))
    \]
    is controllable with respect to $L$ and $E_u$, i.e.,
    \begin{align*}
      \supCC(K, L, (E_{k,u}, E_{1+k,u}, E_{2+k,u}))
      \subseteq \supC(K, L, E_{u})\,.
    \end{align*}
  \end{theorem}
  \begin{proof}
    It is sufficient to show that
    \[
      \supcc := \supCC(K, L, (E_{k,u}, E_{1+k,u}, E_{2+k,u}))
    \]
    is controllable with respect to $L = L_1\| L_2 \| L_k$ and $E_u$. Notice that there exist $\supc_k\subseteq E_k$, $\supc_{1+k}\subseteq E_{1+k}$, and $\supc_{2+k}\subseteq E_{2+k}$ as defined in Corollary~\ref{cor1} so that
    \[
      \supcc = \supc_k \| \supc_{1+k} \| \supc_{2+k}\,.
    \]
    In addition, we know that
    \begin{itemize}
      \item $\supc_k$ is controllable with respect to $L_k$ and $E_{k,u}$,
      \item $\supc_{1+k}$ is controllable with respect to $L_1 \| \supc_k$ and $E_{1+k,u}$,
      \item $\supc_{2+k}$ is controllable with respect to $L_2 \| \supc_k$ and $E_{2+k,u}$\,.
    \end{itemize}
    By Proposition~4.6 in \cite{FLT} (since all the languages under consideration are prefix-closed)
    \[
      \supcc=\supc_k \| \supc_{1+k} \| \supc_{2+k}
    \]
    is controllable with respect to
    \[
      L_k \| (L_1 \| \supc_k) \| (L_2 \| \supc_k) = L\|\supc_k
    \]
    and $E_u$. Analogously, we can obtain that $L\|\supc_k$ is controllable with respect to $L\|L_k = L$ and $E_u$. Finally, by the transitivity of controllability, Lemma~\ref{lem_trans}, we obtain that $\supcc$ is controllable with respect to $L$ and $E_u$, which was to be shown.
  \end{proof}

  The previous theorem demonstrates that the result of our approach shown in Theorem~\ref{thm2} is always controllable with respect to $L$ and $E_u$. Now, we show that if some additional conditions are also satisfied, then the resulting supremal conditionally controllable sublanguage constructed in Theorem~\ref{thm2} is also optimal, i.e., it coincides with the supremal controllable sublanguage of $K$ with respect to $L$ and $E_u$.

  The following result concerning observer properties is proven in \cite[Proposition~4.5]{FLT}.
  \begin{lem}\label{lem45}
    Let $L_i\subseteq E_i^*$, $i=1,2$, be two (prefix-closed) languages, and let $P_i: (E_1\cup E_2)^*\to E_i^*$, where $i=1,2,k$ and $E_k\subseteq E_1\cup E_2$, be natural projections. If $E_1\cap E_2 \subseteq E_k$ and $P^{i}_{k\cap i}$ is an $L_i$-observer, for $i=1,2$, then the projection $P_k$ is an $L_1\|L_2$-observer.
  \end{lem}

  In the following lemma, we prove that conditions of Theorem~\ref{thm2} imply that the projection $P_k$ is OCC for $L$.
  \begin{lem}\label{lem46}
    Let $L_i\subseteq E_i^*$, $i=1,2$, be two (prefix-closed) languages, and let $P_i: (E_1\cup E_2)^*\to E_i^*$, where $i=1,2,k$ and $E_k\subseteq E_1\cup E_2$, be natural projections. Denote by $E_u\subseteq E_1\cup E_2$ the set of uncontrollable events. If $E_1\cap E_2 \subseteq E_k$ and $P^{i+k}_k$ is OCC for $(P^{i+k}_i)^{-1}(L_i)$, for $i=1,2$, then the natural projection $P_k$ is OCC for $L=L_1\|L_2\|L_k$.
  \end{lem}
  \begin{proof}
    Let $s\in L$ be of the form $s=s'\sigma_0\sigma_1\dots\sigma_{k-1}\sigma_k$, for some $k\ge 1$, and assume that $\sigma_0,\sigma_k\in E_{k}$, $\sigma_i\in E\setminus E_k$, for $i=1,2,\dots,k-1$, and $\sigma_k\in E_u$. We need to show that $\sigma_i\in E_u$, for all $i=1,2,\dots,k-1$. However, $P_{i+k}(s) = P_{i+k}(s')\sigma_0P_{i+k}(\sigma_1\dots\sigma_{k-1})\sigma_k\in (P^{i+k}_i)^{-1}(L_i)$ and the OCC property implies that $P_{i+k}(\sigma_1\dots\sigma_{k-1})\in E_u^*$, for $i=1,2$. Consider $\sigma\in\{\sigma_1,\sigma_2,\dots,\sigma_{k-1}\}$. Then, $\sigma\in (E_1\cup E_2)\setminus E_k$. Without loss of generality, assume that $\sigma\in E_1$. Then, $P_{1+k}(\sigma)=\sigma\in E_u$ and $P_{2+k}(\sigma)=\eps\in E_u^*$. Thus, $\{\sigma_1,\sigma_2,\dots,\sigma_{k-1}\}\subseteq E_u$, which was to be shown.
  \end{proof}

  \begin{theorem}\label{thm4}
    Consider the setting of Corollary \ref{cor1}. If, in addition, $L_k\subseteq P_k(L)$ and $P_{i+k}$ is OCC for the language $P_{i+k}^{-1}(L_i\|L_k)$, for $i=1,2$, then
    \begin{align*}
      \supCC(K, L, (E_{k,u}, E_{1+k,u}, E_{2+k,u})) = \supC(K, L, E_{u})\,.
    \end{align*}
  \end{theorem}
  \begin{proof}
    The inclusion $\subseteq$ is proven in Theorem~\ref{thm3}. Thus, we prove the other inclusion.

    From the assumptions,
    \[
      P^{i+k}_k \text{ is the } (P^{i+k}_i)^{-1}(L_i)\text{-observer, for }i=1,2,
    \]
    and
    \[
      P_k^k\text{ is an }L_k\text{-observer}
    \]
    because the observer property always holds for the identity projection.

    Now, Lemma~\ref{lem45} applied to projections $P^{1+k}_k$ and $P^{2+k}_k$ implies that
    \begin{center}
      $P_k$ is an $(P^{1+k}_1)^{-1}(L_1) \| (P^{2+k}_2)^{-1}(L_2) = L_1\|L_2$-observer.
    \end{center}
    Another application of this lemma to projections $P_k$ and $P^k_k$ implies that
    \begin{center}
      $P_k$ is an $(L_1\|L_2)\|L_k = L$-observer.
    \end{center}

    In addition, by Lemma~\ref{lem46}, the projection $P_k$ is also OCC for $L$. For short, denote
    \[
      \supc := \supc(K,L,E_u)\,.
    \]
    We now prove that $P_k(\supc)$ is controllable with respect to $L_k$ and $E_{k,u}$. To do this, assume that $t\in P_k(\supc)$, $a\in E_{k,u}$, and $ta\in L_k\subseteq P_k(L)$. Then, there exists $s\in\supc$ such that $P_k(s)=t$. As $P_k$ is the $L$-observer, there exists $v\in E^*$ such that $sv\in L$ and
    \[
      P_k(sv)=P_k(s)P_k(v)=ta\,,
    \]
    i.e., $v=ua$, for some $u\in (E\setminus E_k)^*$. Furthermore, from the OCC property of $P_k$,
    \[
      u\in E_{u}^*\,.
    \]
    From controllability of $\supc$ with respect to $L$ and $E_u$, this implies that $sua\in\supc$, which means that $P_k(sua)=ta\in P_k(\supc)$. Hence, (i) of Definition~\ref{def:conditionalcontrollability} is satisfied.

    Next, we have that
    \begin{center}
      $P_{i+k}^{i+k}$ (identities) are the $(P^{i+k}_i)^{-1}(L_i)$-observers, for $i=1,2$,
    \end{center}
    and that
    \begin{center}
      $P^{i+k}_{j+k}=P^{i+k}_k$ is the $(P^{i+k}_i)^{-1}(L_i)$-observer, for $\{i,j\}=\{1,2\}$, and \\
      $P_k^k=P^k_{i+k}$ is the $L_k$-observer, for $i=1,2$.
    \end{center}
    Then, similarly as above, Lemma~\ref{lem45} applied to projections $P^{i+k}_{i+k}$, $P^{j+k}_{i+k}$, $j\neq i$, and $P^k_{i+k}$ implies that the projections
    \begin{center}
      $P_{i+k}$ are $L$-observers, for $i=1,2$.
    \end{center}

    Thus, to prove (ii) of Definition~\ref{def:conditionalcontrollability}, assume that, for some $1\le i\le 2$,
    \begin{itemize}
      \item $t\in P_{i+k}(\supc)$,
      \item $a\in E_{i+k,u}$, and
      \item $ta\in L_i\|P_k(\supc)\|P^{j+k}_k(L_j\|P_k(\supc))$, for $j\neq i$.
    \end{itemize}
    Then, there exists $s\in\supc$ such that $P_{i+k}(s)=t$. As $P_{i+k}$ is the $L$-observer, and
    \[
      L_i\|P_k(\supc)\|P^{j+k}_k(L_j\|P_k(\supc))\subseteq P_{i+k}(L) = L_i\|L_k\|P^{j+k}_k(L_j\|L_k),\ j\neq i\,,
    \]
    because
    \[
      P_k(\supc)\subseteq P_k(K)\subseteq P_k(L)\subseteq L_k\,,
    \]
    there exists $v\in E^*$ such that $sv\in L$ and
    \[
      P_{i+k}(sv)=P_{i+k}(s)P_{i+k}(v)=ta\,,
    \]
    i.e., $v=ua$, for some $u\in (E\setminus E_{i+k})^*$. Since $P_{i+k}$ is OCC for $P_{i+k}^{-1}(L_i\|L_k)$ and $sua\in L\subseteq P_{i+k}^{-1}(L_i\|L_k)$, we obtain that $u\in E_{u}^*$. Finally, from the controllability of $\supc$ with respect to $L$ and $E_u$, we obtain that $sua\in\supc$. This means that $P_{i+k}(sua)=ta\in P_{i+k}(\supc)$, which was to be shown.
  \end{proof}

  \begin{rem}
    Note that it is sufficient to assume that $P_{i+k}$ is OCC for $L$. This assumption is less restrictive than the one used in the theorem. Unfortunately, we do not know how to verify this property without computing the whole plant $L$. On the other hand, if $P_{i+k}$ is OCC for $P_i^{-1}(L_i)$, for $i=1,2$, then the theorem holds as well.

    Furthermore, for the verification of $L_k\subseteq P_k(L)$, we can use the property that $P_k(L) = P_k(L_1) \cap P_k(L_2) \cap L_k \subseteq L_k$. Thus, $L_k\subseteq P_k(L)$ if and only if $L_k\subseteq P_k(L_i)$, for $i=1,2$.
  \end{rem}

\subsection{An example}
  In this section, we demonstrate our approach on an example. To do this, let $G=G_1\|G_2$ be a system defined over an event set $E=\{a_1,a_2,c,u,u_1,u_2\}$ as a synchronous composition of systems $G_1$ and $G_2$ defined in Figure~\ref{figA}, where the set of uncontrollable events is $E_u=\{u,u_1,u_2\}$.
  \begin{figure}[ht]
    \centering
    \subfloat[Generator $G_1$.]{\label{figL1}\includegraphics[scale=.9]{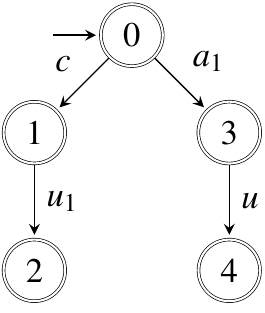}}
    \qquad\qquad
    \subfloat[Generator $G_2$.]{\label{figL2}\includegraphics[scale=.9]{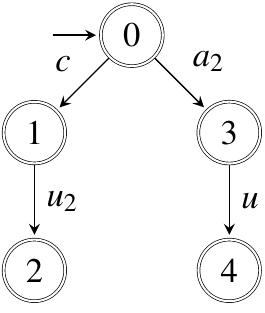}}
    \qquad\qquad
    \subfloat[Coordinator.]{\label{figLk}\includegraphics[scale=.9]{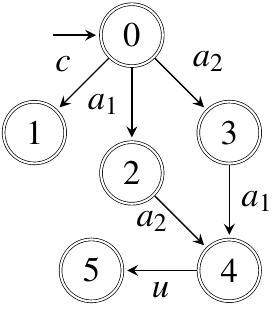}}
    \caption{Generators for $G_1$, $G_2$, and the coordinator.}
    \label{figA}
  \end{figure}
  The behaviors of these systems follow.
  \[
    L(G_1) = \overline{\{cu_1,a_1u\}},\ L(G_2) = \overline{\{cu_2,a_2u\}}
  \]
  and
  \[
    L(G) = \overline{\{a_1a_2u,a_2a_1u,cu_1u_2,cu_2u_1\}}\,.
  \]
  The specification language
  \[
    K=\overline{\{a_2a_1,a_1a_2u,cu_1u_2,cu_2u_1\}}
  \]
  is defined by the generator in Figure~\ref{figK}.
  \begin{figure}[ht]
    \centering
    \includegraphics{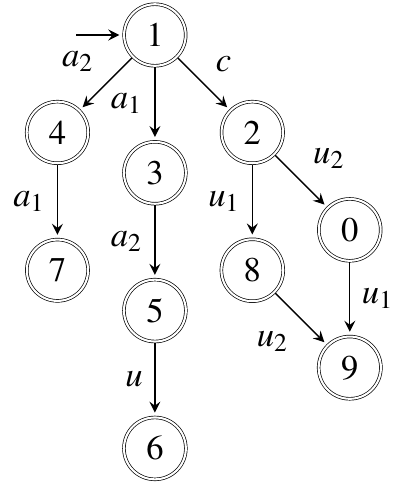}
    \caption{Generator for the specification language $K$.}
    \label{figK}
  \end{figure}

  Now, we need to find a coordinator, i.e., specifically its event set $E_k$. Note that $E_k$ has to contain both shared events $c$ and $u$. In addition, to make $K$ conditionally decomposable, at least one of $a_1$ and $a_2$ has to be added to $E_k$. Thus, we have ensured that $K$ is conditionally decomposable.

  Furthermore, the natural projections must satisfy observer and OCC properties. If $a_i\notin E_k$, for some $i\in\{1,2\}$, then $P^{i+k}_i$ is not OCC for $(P^{i+k}_i)^{-1}(L_i)$. Thus,
  \[
    E_k=\{a_1,a_2,c,u\}\,.
  \]

  Moreover, as we consider prefix-closed languages in this paper, and the coordinator plays a role in blocking issues, we choose the coordinator so that its behavior $L_k$ does not change the original system when composed together, i.e.,
  \[
    L(G_1\|G_2)\|L_k = L(G_1\|G_2)
  \]
  is satisfied, see Figure~\ref{figA}. Our choice is thus
  \[
    L_k = L(P_{1\cap k}^1(G_1)\|P_{2\cap k}^2(G_2))\,,
  \]
  which means that $L_k = \overline{\{c,a_1a_2u,a_2a_2u\}}$. The projections of $K$ are then the following languages:
  \begin{itemize}
    \item $P_k(K)     = \overline{\{a_2a_1,c,a_1a_2u\}}$,
    \item $P_{1+k}(K) = \overline{\{a_1a_2u,a_2a_1,cu_1\}}$, and
    \item $P_{2+k}(K) = \overline{\{a_1a_2u,a_2a_1,cu_2\}}$.
  \end{itemize}
  As mentioned above, it can be verified that the natural projections $P^{i+k}_k$ are $(P_i^{i+k})^{-1}(L_i)$-observers and OCC for the same language, for $i=1,2$. Therefore, we can compute the languages
  \begin{itemize}
    \item $\supc_k     = \overline{\{a_2,c,a_1a_2u\}}$,
    \item $\supc_{1+k} = \overline{\{a_1a_2u,a_2,cu_1\}}$,
    \item $\supc_{2+k} = \overline{\{a_1a_2u,a_2,cu_2\}}$,
  \end{itemize}
  as defined in Theorem~\ref{thm2}, whose synchronous product
  \[
    \supc_k \| \supc_{1+k} \| \supc_{2+k} = \overline{\{a_1a_2u,a_2,cu_1u_2,cu_2u_1\}}
  \]
  is the supremal conditionally controllable sublanguage of $K$, which is controllable by Theorem~\ref{thm3}. However, it can be verified that in this case the resulting language coincides with the supremal controllable sublanguage of $K$ with respect to $L(G)$ and $E_u$. Thus, using our approach, we have computed not only a controllable sublanguage of $K$, but the supremal one.

  Finally, note that the languages involved are not mutually controllable \cite{KvSGM08}, therefore the approach discussed in \cite{KvSGM08} cannot be used in this situation.

\section{Conclusion}\label{sec:conclusion}
  We have considered supervisory control of modular discrete-event systems with global specification languages. A coordination control framework has been adopted where, unlike the purely decentralized setting, a global layer with a coordinator acting on a subset of the global event set has been added. Based on this framework, two main results have been presented. First, a necessary and sufficient condition on a specification language to be exactly achieved in the coordination control architecture, called conditional controllability, has been proposed. Then, it has been shown how the supremal conditionally controllable sublanguage can be synthesized. Finally,  the relationship between supremal conditionally controllable sublanguages and supremal  controllable sublanguage has been investigated.

  In this paper, we have only been interested in the optimality of the control scheme, but blocking that is inherent to modular and, more generally, to our coordinated control synthesis has not been considered. It is then sufficient to choose a suitable coordinator event set and the coordinator itself need not impose any restriction on the behavior because its supervisor can take care of a required  restriction of the plant projected to the coordinator events. In a future work, however, it is our plan to address the blocking issue by considering a suitable coordinator and combine it with the three supervisors so that both blocking and maximal permissivity are handled at the same time within the coordination scheme.

  Thus, more work on coordination control dealing with global specification languages is needed. In particular, the synthesis of coordinators for nonblockingness is to be developed and the approach should be extended to partially observed modular plants.

\section*{Acknowledgment}
  A comment by Klaus Schmidt (U. Erlangen) is herewith gratefully acknowledged.

  This work has been supported by the EU.ICT 7FP project DISC no.~224498, and by the Academy of Sciences of the Czech Republic, Institutional Research Plan no. AV0Z10190503.

\bibliographystyle{plain}
\bibliography{supcc}

\end{document}